\documentclass[12pt, twoside, leqno]{article}

\usepackage{amsmath,amsthm}
\usepackage{amssymb}

\usepackage{enumerate}

\usepackage{graphicx}

\newtheorem{thm}{Theorem}[section]

\newtheorem{lem}[thm]{Lemma}
\newtheorem{prop}[thm]{Proposition}

\theoremstyle{definition}

\numberwithin{equation}{section}

\newcommand{\mr}{\ensuremath{\mathbb R}}
\newcommand{\mh}{\ensuremath{\mathbb H}}

\newcommand{\fd}{\ensuremath{\mathfrak d}}
\newcommand{\fp}{\ensuremath{\mathfrak p}}

\newcommand{\fm}{\ensuremath{\mathfrak m}}
\newcommand{\fn}{\ensuremath{\mathfrak n}}
\newcommand{\ri}{\ensuremath{\mathcal{O}}} 
\newcommand{\id}{\ensuremath{\mathcal{I}}} 

\newcommand{\shortmod}{\ensuremath{\negthickspace \negthickspace \negthickspace \pmod}}

\newcommand{\sumstar}{\sideset{}{^*}\sum}

\begin{document}


\baselineskip=17pt


\title{Low-lying zeros for $L$-functions associated to Hilbert modular forms of large level}

\author{Sheng-Chi Liu\\
Department of Mathematics and Statistics\\
Washington State University\\
Pullman, WA 99164, USA\\
E-mail: scliu.wsu@gmail.com
\and
Steven J. Miller\\
Department of Mathematics and Statistics\\
Williams College\\
Williamstown, MA 01267, USA\\
E-mail: sjm1@williams.edu, Steven.Miller.MC.96@aya.yale.edu}

\date{}

\maketitle


\renewcommand{\thefootnote}{}



\footnote{The first author was supported by a grant from the Simons Foundation (\#344139) and the second by a grant from the NSF (DMS1561945) and Washington State University.}

\renewcommand{\thefootnote}{\arabic{footnote}}
\setcounter{footnote}{0}


\begin{abstract}
We determine the 1-level density of families  of Hilbert modular forms, and show the answer agrees only with orthogonal random matrix ensembles.
\end{abstract}

\maketitle


\section{Introduction}

\subsection{History}

Since the work of Montgomery and Odlyzko, there has been a large body of literature on the similarities in behavior of zeros of $L$-functions and the eigenvalues of random matrices. Assuming the Generalized Riemann Hypothesis (GRH), the non-trivial zeros of automorphic $L$-functions have real part 1/2, and thus these zeros can be ordered (and many have hoped to find a spectral interpretation). The earliest  statistics studied were the $n$-level correlations and spacings between adjacent zeros (see \cite{Hej, Mo, Od1, Od2, RS}). While the number theory agreed with the behavior of eigenvalues from the Gaussian Unitary Ensemble, these statistics are insensitive to the removal of finitely many zeros. In particular, these statistics cannot detect any information about the behavior of zeros at or near the central point. This is unfortunate, as this is where the most important arithmetic lives (for example, the Birch and Swinnerton-Dyer conjecture for elliptic curves).

A major breakthrough came with the work of Katz and Sarnak \cite{KS1,KS2}. They proved that while many random matrix ensembles have the same $n$-level correlations, they showed there is another statistic, the $n$-level density, where each ensemble has a different answer. The major difference between these two statistics is that in the $n$-level correlation all the zeros contribute equally, while in the $n$-level density most of the contribution comes from the zeros near or at the central point, making it an ideal quantity to investigate the arithmetic of families. It is important to note that this statistic, which we define below, is for a family of $L$-function. While an individual $L$-function has infinitely many zeros on the critical line (and thus we have the ability to perform averages over them), it is expected that there are only a bounded number within a few multiples of the analytic conductor of the central point. Thus, in order to be able  to average, we look at a \emph{family} of $L$-functions with similar properties. Examples include $L$-functions attached to Dirichlet characters, elliptic curves, modular forms, and Maass forms, as well as symmetric lifts and Rankin-Selberg convolutions of the above, to name just a few; in this paper we study Hilbert modular forms. There is now a vast literature on these and related  families showing agreement between number theory and random matrix theory; see the seminal paper \cite{ILS} for the first work on the subject and \cite{IK} for  a review of needed properties of $L$-functions, the survey articles \cite{BFMT-B, MMRT-BW} and the references therein for a collection of some results and methods, and \cite{DM,SaShTe} for a discussion on how to determine the symmetry group associated to a family of  $L$-functions.

We define the 1-level density of an $L$-function $L(s,f)$ with zeros $1/2 + i \gamma_f$ and analytic conductor $c_f$ by
\begin{equation}
 D(f; \phi) \ := \  \sum_{\gamma_f} \phi \left( \frac{\gamma_f}{2\pi} \log c_f \right),
\end{equation}
where $\phi(x)$ is an even Schwartz function such that its Fourier transform
\begin{equation*}
\widehat{\phi} (y)\ := \  \int_{-\infty}^{\infty} \phi(x) e^{-2\pi i xy}dx
\end{equation*}
has compact support (and thus $\widehat{\phi}(y)$ extends to an entire function). Note  that if GRH is true then each $\gamma_f \in \mathbb{R}$ and the non-trivial zeros are all  on the critical line ${\rm Re}(s) = 1/2$. As $\phi$ is Schwartz, it is of rapid decay and most of the contribution is from the zero near the central point, as desired. The 1-level density of a family $\mathcal{F}$ is a the average\footnote{Frequently one considers a weighted  average, where the weights facilitate applying summation formulas. For example, the harmonic or Petersson weights are frequently used for families of cuspidal newforms, which assist in applying the Petersson formula to evaluate.} of $D(f,\phi)$ over  $f\in\mathcal{F}$; one typically breaks the family $\mathcal{F}$ into sub-families $\mathcal{F}_N$ of the same conductor $N$ (or conductor in some interval, say $[N, 2N)$, and then let $N\to\infty$. One can similarly define the $n$-level density (see \cite{ILS}, where we have a test function of $n$-variables and sum over zeros $\gamma_{f;i_1}, \dots, \gamma_{f;i_n}$ such that $i_j \neq \pm i_\ell$ if $j \neq \ell$) or the $n$\textsuperscript{th} centered moment (see \cite{HM}).\footnote{This condition in the $n$-level density simplifies the resulting expression and leads to the clean determinant expansions; via inclusion-exclusion it is equivalent to the $n$\textsuperscript{th} centered moment if the test function $\phi$ satisfies $\phi(x_1,\dots,x_n) = \prod_{\ell=1}^n \phi(x_\ell)$.}

The Katz-Sarnak density conjecture states that as the conductors tend to infinity,  the $n$-level density of a  family of $L$-functions converges to the $n$-level density of the eigenvalues of a sub-ensemble of classical compact matrices with size  tending to infinity: \begin{equation}
\lim_{N\to\infty} D_n(\mathcal{F}_N,\phi) \ = \ \int_{-\infty}^\infty\cdots\int_{-\infty}^\infty \phi(x_1,\dots,x_n)W_{n,\mathcal{G}(\mathcal{F})}(x_1,\dots,x_n)dx_1\cdots dx_n,\end{equation} where $\mathcal{G}(\mathcal{F})$ is the symmetry group associated to the family of $L$-functions, and
\begin{center}
\begin{tabular}{cc} \\
$\underline{\ \ \ \ \ \ \ \ \ \ \ \ \mathcal{G} \ \ \ \ \ \ \ \ \
\ \ }$ & $\underline{\ \ \ \ \ \ \ \ \ \ \ \ W_{n,\mathcal{G}} \
\ \ \ \ \ \ \ \ \ \ \ }$ \\ \\
${\rm U(N)}$ & $\det\left(K_0(x_j,x_k)\right)_{1 \leq
j, k
\leq n}$ \\ \\
${\rm Sp(N)}$ & $\det\left(K_{-1}(x_j,x_k)\right)_{1 \leq j, k \leq
n}$ \\ \\
${\rm SO(even)}$ & $\det\left(K_{1}(x_j,x_k)\right)_{1
\leq j, k \leq n}$ \\ \\
${\rm SO(odd)}$ & $\det\left(K_{-1}(x_j,x_k)\right)_{1 \leq j, k \leq
n}$ \\ \\  
$ $ & $\ \ +\ \sum_{\nu=1}^n \delta(x_\nu) \det\left(K_{-1}(x_j,x_k)\right)_{1
\leq j, k \neq \nu \leq n}$
\end{tabular}
\end{center}
where
\begin{eqnarray}
K_\epsilon(x,y)\ =\ \frac{\sin\left(\pi(x-y) \right)}{\pi(x-y)} +
\epsilon \frac{\sin\left(\pi(x+y) \right)}{\pi(x+y)}
\end{eqnarray} and ${\rm O}$ is the average of ${\rm SO(even)}$ and ${\rm SO(odd)}$.

It is often convenient to look on the Fourier transform side; for the $1$-level densities we have
\begin{eqnarray}
\widehat{W_{1,{\rm SO(even)}} }(u) &\ =\ & \delta_0(u) + \frac12 \eta(u)
\nonumber\\ \widehat{W_{1,{\rm O}} }(u) & = & \delta_0(u) + \frac12
\nonumber\\ \widehat{W_{1,{\rm SO(odd)}} }(u) & = & \delta_0(u) - \frac12
\eta(u) + 1 \nonumber\\ \widehat{W_{1,{\rm Sp}} }(u) & = & \delta_0(u) -
\frac12 \eta(u) \nonumber\\ \widehat{W_{1,{\rm U}} }(u) & = & \delta_0(u)
\end{eqnarray}
where $\eta(u)$ is $1$, $1/2$, and $0$ for $|u|$ less than $1$,
$1$, and greater than $1$, and $\delta_0$ is the standard Dirac
Delta functional.


The five classical compact groups have distinguishable 1-level densities when the support of $\widehat{\phi}$ exceeds $[-1, 1]$; however, for support in $(-1,1)$ the three orthogonal flavors are mutually indistinguishable (though they are different than symplectic and unitary). This motivates many works which try to break the region $[-1,1]$, though this is frequently difficult due to arithmetic obstructions in the resulting sums. Fortunately it is not necessary to identify the underlying group symmetry; in his thesis Miller \cite{Mil1, Mil2} noted  that the  five groups have mutually distinguishable 2-level densities for arbitrarily small support.

\subsection{New Results}

Let $F$ be a totally real number field of degree $n$ over $\mathbb{Q}$ with narrow class number one.
Let $\mathcal{O}$ be the ring of integers, $U$ be the unit group, $D$ be the discriminant, $R$ be the regulator, and $W$ be the number of roots of unity. For $\nu \in F$, let $\nu^{(i)}:= \sigma_i(\nu)$ where
$\sigma_1, \sigma_2, \dots, \sigma_n$ are the real embeddings of $F$.
We denote by $\nu \gg 0$ for a totally positive $\nu \in F$.

Let $\Gamma = SL_2(\mathcal{O})$ be the Hilbert modular group, which acts on the $n$-fold product $\mh^n$ of the complex
upper half-plane $\mathbb{H}$.
For any ideal $\mathcal{I} \subset \mathcal{O}$, let
\begin{equation*}
\Gamma_0(\id) \ := \  \left\{ \begin{pmatrix}
a & b \\ c & d \end{pmatrix} \in \Gamma \ \Big| \ c \equiv 0 \pmod{\id} \right\}.
\end{equation*}
We assume $\id$ is square-free  in this paper and denote the norm of $\id$ by $N(\id)= [ \ri: \id ]$.

For an integer $k \geq 2$, let $H^{\star}_{2k}(\id)$ denote the set of \textit{primitive} Hilbert modular cusp forms of weight $(2k, \dots, 2k)$ and level $\id$. Assuming GRH for $L(s, f)$, we may write the nontrivial zeros of $L(s, f)$ by $\rho_f= \frac{1}{2} + i \gamma_f$ with $\gamma_f$ real.

Our main result is that  the one-level density agrees only with orthogonal symmetry for the family  $f \in H^{\star}_{2k}(\id)$. As our support exceeds $[-1, 1]$ we are able to uniquely identify the corresponding symmetry group, and there is no need to compute the 2-level density.

\begin{thm}
Fix any Schwartz function $\phi$ with  $\mathrm{supp}(\widehat{\phi}) \subset  (-\frac{3}{2}, \frac{3}{2})$.
Assume GRH for $L(s,f)$ and $L(s, \mathrm{sym}^2f)$ for all $f \in H^{\star}_{2k}(\id)$.
Then we have
\begin{equation}
\lim_{N(\id) \to \infty} \frac{1}{|H^{\star}_{2k}(\id)|} \sum_{f \in H^{\star}_{2k}(\id)} D(f; \phi) \ = \  \int_{-\infty}^{\infty} \phi (x) W(O) (x) dx
\end{equation}
where $\id$ runs over square-free ideals and $W(O)=1 + \frac{1}{2} \delta_0(x)$ and $c_f= k^{2n} N(I)$ is the analytic conductor.
\end{thm}

\section{Preliminaries}
For an integer $k \geq 2$, let $S_{2k}(\id)$ be the  space of cuspidal Hilbert modular forms
of weight $(2k, \dots, 2k)$ for $\Gamma_0(\id)$ (see \cite{Ga}). This is the space of holomorphic functions $f(z)$ on $\mh^n$ which vanish in the cusps of
$\Gamma_0(\id)$ and satisfy
\begin{align*}
f(\gamma z)\ = \  N(cz + d)^{2k}f(z) \quad \textrm{for} \quad \gamma \ = \
        \left( \begin{array} {ccc}
         a & b \\
         c & d
         \end{array} \right)
      \in \Gamma_0(\id),
\end{align*}
where for $z=(z_1, \dots, z_n) \in \mh^n$ we have $$ N(cz + d) \ =\ \prod_{i=1}^{n}(\sigma_i(c) z_i +  \sigma_i(d)).$$

It was shown by Shimizu \cite{Sh} that $$\mathrm{dim} (S_{2k}(\id))\ \sim\ \mathrm{vol}(\Gamma_0(\id) \backslash \mh^n) \frac{(2k-1)^n}{(4\pi)^n} $$ as $kN(\id) \to \infty$.

For each $f \in S_{2k}(\Gamma_0(\id))$ one has the Fourier expansion
$$ f(z)\ = \  \sum_{\substack{ \nu \in \ri \\ \nu \gg 0}} a_f(\nu) N(\nu)^{\frac{2k-1}{2}} e^{2\pi i \mathrm{Tr}(\nu \delta^{-1}z)},
$$
where $\fd= (\delta)$ with $\delta \gg 0$ the different of $F$ and
$$ \mathrm{Tr}(\nu \delta^{-1}z) \ = \  \sum_{i=1}^n \sigma_i(\nu \delta^{-1}z).
$$
Note that the Fourier coefficients $a_f(\nu)$ depend only on the ideal $\fm= (\nu)$. Hence we also denote $a_f(\mathfrak{m}) := a_f(\nu)$.

The Petersson inner product of two forms $f, g \in S_{2k}(\id)$ is defined by
\begin{equation*}
\left\langle f, g \right\rangle_{\id}\ = \  \int_{\Gamma_0(\id) \backslash \mh^n}  f(z) \overline{g(z)} \prod_{i=1}^{n} y_i^{2k}\frac{dx_idy_i}{y_i^2},
\end{equation*}
where $z = x + iy = (x_1+iy_1, \dots, x_n+iy_n)$.

Let $B_{2k}(\id)$ be an orthogonal basis of $S_{2k}(\id)$.
Let $$w_f\ = \  \frac{(4\pi)^{n(2k-1)}}{\Gamma^n(2k-1)D^{2k-1/2}} \left\langle f, f \right\rangle_{\id},$$ and let
\begin{equation*}
\Delta_{2k, \id} (\fm, \mathfrak{n}) \ := \  \sum_{f \in B_{2k}(\id)} w_f^{-1} \overline{a_f (\fm) } a_f (\fn).
\end{equation*}
 The following Petersson trace formula was proved by Luo \cite{Lu}.

\begin{prop} \label{Petersson}
For any $\fm= (\nu), \fn= (\mu)$ where $\nu \gg 0$ and $\mu \gg 0$ in $\mathcal{O}$,  we have
\begin{align} \label{Petersson}
\Delta_{2k, \id} (\fm, \fn) \ = \    \chi_{\nu} (\mu) + \frac{(2\pi)^n (-1)^{nk}}{D^{1/2}} \sum_{\epsilon \in U} \sum_{c \in \id^*/U}
 \frac{S(\nu, \mu \epsilon^2; c)}{|N(c)|}  \prod_{i=1}^{n}J_{2k-1} \left(\frac{4\pi \sqrt{\nu^{(i)} \mu^{(i)}} |\epsilon^{(i)}|}{|c^{(i)}|} \right),
\end{align}
where  $\chi_{\nu}$ is the characteristic function of the set $\{ \nu \epsilon^2 : \epsilon \in U \}$, and
\begin{equation*}
 S(\nu, \mu; c)\ = \  \sumstar_{a \shortmod {c}} e \left( \mathrm{Tr} \left(\frac{\nu a + \mu \overline{a}}{c} \right) \right)
\end{equation*}
is a generalized Kloosterman sum.
\end{prop}

We may assume that the $c$'s in (\ref{Petersson})  are chosen satisfying $|N(c)|^{1/n} \ll |c^{(i)}| \ll |N(c)|^{1/n}$.
The Weil bound (see \cite[2.6]{Ve}) asserts that
\begin{equation} \label{WeilBound}
|S(\nu, \mu ; c)|\ \ll\ N(((\nu), (\mu), (c)))^{1/2} \tau((c))^{1/2} N((c))^{1/2}
\end{equation}
where $\tau((c))$ is the generalized divisor function.

Let $H_{2k}^\star (\mathfrak{M})$ be the set of primitive forms (newforms) of weight $(2k, \dots, 2k)$ and level $\mathfrak{M} \subset \ri$.
In view of the generalization of Atkin-Lehner's theory on newforms (see \cite{AL,Mi}), we have the orthogonal decomposition
\begin{equation}
S_{2k}(\id)\ = \ \bigoplus_{\mathfrak{L} \mathfrak{M}= \id} \bigoplus_{f \in H_{2k}^\star (\mathfrak{M})} S_{2k}(\mathfrak{L}; f)
\end{equation}
where $S_{2k}(\mathfrak{L}; f)$ is the linear space spanned by the forms $N(\ell)^{k}f(\ell z)$ with $ (\ell) | \mathfrak{L}$ and $ \ell \gg 0 $ in $\ri$.

If $f \in H_{2k}^\star (\mathfrak{M})$, then $f$ is  simultaneously an eigenfunction of all Hecke operators $T_{\mathfrak{n}}$, where $\mathfrak{n}=(\mu),  \mu \gg 0$ in $\ri$ (see \cite{Ga}).
Let $\lambda_f(\mu)$ or $\lambda (\mathfrak{n})$ denote the eigenvalue of $T_{\mathfrak{n}}$. Then $a_f(\mathfrak{n})= a_f(\ri) \lambda_f(\mathfrak{n})$. We normalize $f$ so that $a_f(\ri)=1$ and hence $a_f(\mathfrak{n})= \lambda_f(\mathfrak{n})$.
The Hecke eigenvalues satisfy the following multiplicative relation:
\begin{equation} \label{eigenvaluemultiplicative}
\lambda_f(\mu) \lambda_f(\nu)\ = \  \sum_{\substack{d | (\mu, \nu), d \gg 0 \\ ((d), \mathfrak{M})= (1)}} \lambda_f \left( \frac{\mu \nu}{ d^2} \right).
\end{equation}
The generalized Petersson-Ramanujan conjecture
\begin{align} \label{Ramanujan}
\lambda_f(\mathfrak{n}) \ll N(\mathfrak{n})^{\varepsilon},
\end{align}
is true for Hilbert modular forms (see Blasius \cite{Bl}).
Moreover,  for $\fp | \mathfrak{M}$
\begin{equation} \label{eigenvaluelevel}
\lambda_f(\fp)^2\ = \  \frac{1}{N(\fp)}.
\end{equation}
For $f \in H_{2k}^\star (\id)$, the $L$-function associated to $f$ is defined by
\begin{align*}
L(s,f) \ = \  \sum_{\substack{ \fm \subset \ri \\ \fm \ne (0)}} \frac{\lambda_f(\fm)}{N(\fm)^s}
\end{align*}
for $\mathrm{Re} (s) > 1$. Its Euler product is
\begin{align*}
  L(s, f)  \ = \  \prod_{\fp | \id}\left( 1- \frac{\lambda_f(\fp)}{N(\fp)^s}  \right)^{-1} \prod_{\fp \nmid \id} \left(   1- \frac{\alpha(\fp)}{N(\fp)^s}\right)^{-1} \left(   1- \frac{\beta(\fp)}{N(\fp)^s}\right)^{-1}.
\end{align*}
The completed $L$-function
\begin{equation*}
\Lambda (s,f) \ := \  (2\pi)^{-ns} \Gamma^n (s+ k -\tfrac{1}{2}) (N(\id) D^2)^{s/2} L(s,f)
\end{equation*}
satisfies the functional equation
\begin{equation*}
\Lambda (s,f)\ = \  \varepsilon_f \Lambda (1-s, f)
\end{equation*}
where $\varepsilon_f =  i^{2nk} \mu(\id) \lambda_f(\id) N(\id)^{1/2} = \pm 1$. Here $\mu(\id)$ is the generalized M\"obius function.

The following lemma generalizes \cite[Lemma 2.5]{ILS}.
\begin{lem} \label{normexpression}
Let $f$ be a newform of weight $(2k, \dots, 2k)$ and level $\mathfrak{M} | \id$. Then we have
\begin{equation}
 \left\langle f, f \right\rangle_{\id} \ = \  2 \left( \frac{\Gamma(2k)}{(4\pi)^{2k}}\right)^n \frac{D^{2k}}{\pi^n} \zeta_F (2) \nu(\id) \phi(\mathfrak{M}) Z(1, f),
\end{equation}
where $\zeta_F (s)$ is the Dedekind zeta function,
\begin{equation*}
\nu (\id)\ = \  [SL_2(\ri): \Gamma_0(\id)]\ = \  N(\id) \prod_{\fp | \id} (1 + \frac{1}{N(\fp)}),
\end{equation*}
\begin{equation*}
\phi(\mathfrak{M})\ = \  \prod_{\fp | \mathfrak{M}} (1- \frac{1}{N(\fp)}),
\end{equation*}
and
\begin{equation*}
Z(s, f)\ = \  \sum_{\mathfrak{n} \subset \ri} \frac{\lambda_f(\mathfrak{n}^2)}{N(\mathfrak{n})^s}.
\end{equation*}
\end{lem}

As the proof follows from the Eisenstein series $\displaystyle E(z, s)=\sum_{\gamma \in \Gamma_{\infty} \backslash \Gamma_0(\id)} N(\mathrm{Im}(\gamma z))^s$ and the standard Rankin-Selberg unfolding method, as in \cite{ILS} (also see \cite[p. 24-25]{vG}),  we omit the details here.  Note that the residue of $E(z, s)$ at $s=1$ is equal to  $\displaystyle \frac{2^{n-1} R \pi^n}{D W \mathrm{vol}(\Gamma_0(\id) \backslash \mh^n)}$.

It will be more convenient to work with the local zeta function
\begin{equation*}
Z_{\id}(s, f)\ := \  \sum_{ \mathfrak{n} | \id^{\infty}} \frac{\lambda_f(\mathfrak{n}^2)}{N(\mathfrak{n})^s}.
\end{equation*}

For $f \in H_{2k}^\star (\mathfrak{M})$ with $\mathfrak{M} | \id$, one deduces by (\ref{eigenvaluemultiplicative}) and (\ref{eigenvaluelevel}) that
\begin{equation}  \label{localzeta}
Z_{\fp}(1, f)\ = \  \begin{cases} (1+ \frac{1}{N(\fp)})^{-1} \rho_f(\fp)^{-1} & \text{if} \ \fp \nmid \mathfrak{M} \\
 (1+ \frac{1}{N(\fp)})^{-1} (1- \frac{1}{N(\fp)})^{-1} & \text{if} \ \fp | \mathfrak{M},
\end{cases}
\end{equation}
where  $\rho_f(\mathfrak{c})$ is the multiplicative function given by
\begin{equation}
\rho_f(\mathfrak{c})\ = \  \sum_{\mathfrak{b} | \mathfrak{c}} \mu(\mathfrak{b}) \left( \frac{\lambda_f(\mathfrak{b})}{\nu(\mathfrak{b})}\right)^2 \ = \  \prod_{\fp | \mathfrak{c}} \left( 1 - N(\fp) \left( \frac{\lambda_f(\fp)}{N(\fp) + 1}\right)^2 \right).
\end{equation}

Define
\begin{equation}
f_{\mathfrak{q}} (z) \ = \  \left( \frac{N(\mathfrak{q})}{\rho_f(\mathfrak{q})}\right)^2 \sum_{\substack{\mathfrak{cd=q} \\ \mathfrak{d} = (\xi_{\mathfrak{d}}), \ \xi_{\mathfrak{d} \gg 0 } } } \mu (\mathfrak{c}) \nu (\mathfrak{c})^{-1} N(\mathfrak{d})^{\frac{2k-1}{2}} f(\xi_{\mathfrak{d}}z).
\end{equation}
 Arguing as in the proof of \cite[Prop. 2.6]{ILS}, one can derive the proposition below.

\begin{prop} \label{orthogonalbasis}
Let $\id= \mathfrak{LM}$ be a squarefree ideal in $\ri$ and $f \in  H_{2k}^{\star}(\mathfrak{M})$. Then $\{ f_{\mathfrak{q}}; \mathfrak{q} | \mathfrak{L} \}$ is an orthogonal basis of $S_{2k}(\mathfrak{L}; f)$. Moreover, $\left\langle f_{\mathfrak{q}}, f_{\mathfrak{q}} \right\rangle_{\id} = \left\langle f, f \right\rangle_{\id}$.
\end{prop}

One deduces the following result from Lemma \ref{normexpression} and Proposition \ref{orthogonalbasis}.

\begin{lem} For $\fm, \fn \subset \ri$ with  $(\fm, \id)= (\fn, \id)= (1)$, we have
\begin{equation} \label{Delta1}
\Delta_{2k, \id} (\fm, \fn)\ = \  \frac{(4\pi)^n \pi^n}{2 D^{1/2} \zeta_F(2)} \cdot \frac{1}{(2k-1)^n N(\id)} \sum_{\mathfrak{LM}= \id} \sum_{f \in H_{2k}^{\star}(\mathfrak{M})} \frac{Z_{\id}(1, f)}{Z(1, f)} \overline{\lambda_f(\fm)} \lambda_f(\fn).
\end{equation}
\end{lem}

Let
\begin{equation}
\Delta_{2k, \mathfrak{M}}^{\star}  (\fm, \fn)\ := \  \sum_{f \in H_{2k}^{\star}(\id)} \frac{Z_{\mathfrak{M}}(1, f)}{Z(1, f)} \overline{\lambda (\fm)} \lambda_f(\fn).
\end{equation}

\begin{prop} For $(\fm, \id) = (\fn, \id) = (1)$,
\begin{equation} \label{Delta2}
\Delta_{2k, \id} (\fm, \fn) \ = \   \frac{(4\pi)^n \pi^n}{2 D^{1/2} \zeta_F(2)} \cdot \frac{1}{(2k-1)^n N(\id)} \sum_{\mathfrak{LM}= \id} \sum_{\mathfrak{l} | \mathfrak{L}^{\infty}} \frac{1}{N(\mathfrak{l})} \Delta_{2k, \mathfrak{M}}^{\star} (\fm \mathfrak{l}^2, \fn),
\end{equation}
and
\begin{equation} \label{DeltaStar}
\Delta_{2k, \id}^{\star} (\fm, \fn)\ = \  \frac{2 D^{1/2} \zeta_F(2)}{ (4\pi)^n \pi^n} (2k-1)^n \sum_{\mathfrak{LM}= \id} \mu(\mathfrak{L}) N(\mathfrak{M}) \sum_{\mathfrak{l} | \mathfrak{L}^{\infty}} \frac{1}{N(\mathfrak{l})} \Delta_{2k, \mathfrak{M}} (\fm \mathfrak{l}^2, \fn).
\end{equation}
\end{prop}

\begin{proof}
We obtain (\ref{Delta2}) from (\ref{Delta1}), (\ref{localzeta}) and (\ref{eigenvaluemultiplicative}), while (\ref{DeltaStar}) follows from (\ref{Delta1}) and M\"obius inversion.
\end{proof}

Let
\begin{equation*}
\Delta_{2k, \id}^{\star} (\fn) \ := \  \sum_{f \in H_{2k}^{\star} (\id)} \lambda_{f} (\fn).
\end{equation*}

\begin{thm} \label{DeltaStarNoWeight}
For $(\fn, \id)= (1)$, we have
\begin{equation}
\Delta_{2k, \id}^{\star} (\fn) \ = \  \frac{2 D^{1/2} \zeta_F(2)}{(4\pi)^n \pi^n} (2k-1)^n \sum_{\mathfrak{LM}= \id} \mu(\mathfrak{L}) N(\mathfrak{M}) \sum_{(\fm, \mathfrak{M})= (1)} \frac{1}{N(\fm)} \Delta_{2k, \mathfrak{M}} (\fm^2, \fn).
\end{equation}
\end{thm}

\begin{proof}
One verifies directly that $\displaystyle \Delta_{2k, \id}^{\star} (\fn) = \sum_{(\fm, \id)= (1)} \frac{1}{N(\fm)} \Delta_{2k, \id}^{\star} (\fm^2, \fn)$ and the result follows from (\ref{DeltaStar}).
\end{proof}

Let $X, Y \geq 1$, be two parameters. We will choose $X$ and $Y$ as a small power of $kN(\id)$ later. We write
\begin{equation*}
\Delta_{2k, \id}^{\star} (\fn) \ := \  \Delta_{2k, \id}' (\fn) + \Delta_{2k, \id}^{\infty} (\fn),
\end{equation*}
where
\begin{equation}
\Delta_{2k, \id}' (\fn) \ := \  \frac{2 D^{1/2} \zeta_F(2)}{(4\pi)^n \pi^n} (2k-1)^n \sum_{\substack{\mathfrak{LM}= \id \\ N(\mathfrak{L}) \le X}} \mu(\mathfrak{L}) N(\mathfrak{M}) \sum_{\substack{ (\fm, \mathfrak{M})= (1) \\ N(\fm) \le Y}} \frac{1}{N(\fm)} \Delta_{2k, \mathfrak{M}} (\fm^2, \fn)
\end{equation}
and $\Delta_{2k, \id}^{\infty} (\fn)$ is the complementary sum.

Assume the sequence $\{a_{\mathfrak{q}} \}_{\mathfrak{q} \subset \ri}$ (indexed by ideals) satisfying
\begin{equation} \label{sequencecondition}
\sum_{(\mathfrak{q}, \fn \id)\ = \  (1)} \lambda_{f} (\mathfrak{q}) a_{\mathfrak{q}}\ \ll\ (N(\fn \id) k)^{\varepsilon}
\end{equation}
for all $f \in H_{2k}^{\star} (\mathfrak{M})$ with $\mathfrak{M} | \id$ such that the implied constant depends only on $\varepsilon$.
We will choose $a_{\mathfrak{q}}= N(\fp)^{-1/2} \log N(\fp)$ if $\mathfrak{q} = \fp, \ N(\fp) \le Q$ and $a_{\mathfrak{q}}= 0$ otherwise, provided $\log Q \ll \log (kN(\id))$; this choice satisfies (\ref{sequencecondition}).

\begin{lem} \label{Deltainfty}
Suppose $(\fn, \id)=1$. For any sequence $\{a_{\mathfrak{q}} \}_{q \subset \ri}$ satisfies $(\ref{sequencecondition})$, we have
\begin{equation}
\sum_{(\mathfrak{q}, \fn \id)= (1)} \Delta_{2k, \id}^{\infty} (\fn \mathfrak{q}) a_{\mathfrak{q}}\ \ll\ k^{n} N(\id) (X^{-1} + Y^{-1/2}) (N(\fn \id)k XY)^{\varepsilon}.
\end{equation}
\end{lem}

\begin{proof}
By Theorem \ref{DeltaStarNoWeight} and (\ref{Delta2}), we have
\begin{equation*}
\Delta_{2k, \id}^{\infty} (\fn \mathfrak{q}) \ = \  \sum_{\substack{\mathfrak{KLM}= \id \\ N(\mathfrak{L}) > X}} \mu (\mathfrak{L}) \sum_{f \in H_{2k}^{\star}(\mathfrak{M})} \lambda_{f} (\fn \mathfrak{q}) +  \sum_{\substack{\mathfrak{KLM}= \id \\ N(\mathfrak{L}) \le X}} \mu (\mathfrak{L}) \sum_{f \in H_{2k}^{\star}(\mathfrak{M})} \lambda_{f} (\fn \mathfrak{q}) R_f(\mathfrak{KM} ; Y)
\end{equation*}
where
\begin{equation*}
R_{f} (\mathfrak{KM}; Y) \ = \  \frac{Z_{\mathfrak{KM}} (1, f)}{Z (1, f)} \sum_{\substack{(\fm, \mathfrak{KM}) = (1) \\ N(\fm) > Y}} \frac{1}{N(\fm)} \lambda_f(\fm^2).
\end{equation*}
By GRH for $L(s, \mathrm{sym}^2f)$, we have
\begin{equation*}
R_f(\mathfrak{KM}; Y)\ \ll\ Y^{-1/2} (k N(\mathfrak{KM}Y))^{\varepsilon}.
\end{equation*}
Moreover, we have $|\lambda_f(\fn)| \ll N(\fn)^{\varepsilon}$. Hence the lemma follows by the above estimates and (\ref{sequencecondition}).
\end{proof}

\begin{prop} We have
\begin{equation} \label{dimension}
|H_{2k}^{\star}(\id)| \ = \  \frac{2D^{1/2} \zeta_F(2)}{(4\pi)^n \pi^n} (2k-1)^n N(\id) \prod_{\fp | \id} \left( 1- \frac{1}{N(\fp)}\right) + O( k^{\frac{3}{5}n} N(\id)^{\frac{3}{5}} (kN(\id))^{\varepsilon}).
\end{equation}
\end{prop}

\begin{proof} Note
\begin{equation*}
|H_{2k}^{\star}(\id)| \ = \  \Delta_{2k, \id}^{\star} (\ri) \ = \  \Delta_{2k, \id}' (\ri) + \Delta_{2k, \id}^{\infty} (\ri).
\end{equation*}
By Lemma \ref{Deltainfty}, $\Delta_{2k, \id}^{\infty} (\ri) \ll k^n N(\id) (X^{-1} + Y^{-1/2}) (N(\id)kXY)^{\varepsilon}$.
By Proposition \ref{Petersson},
\begin{equation*}
\Delta_{2k, \id}' (\ri) \ = \  \frac{2D^{1/2} \zeta_F(2)}{(4\pi)^n \pi^n} (2k-1)^n \sum_{\substack{\mathfrak{LM}= \id \\ N(\mathfrak{L}) \le X}} \mu(\mathfrak{L}) N(\mathfrak{M}) + \frac{2(-1)^{nk} \zeta_F(2)}{(2\pi)^n} (2k-1)^n  E_{2k, \id}
\end{equation*}
where
\begin{equation*}
 E_{2k, \id}\ = \  \sum_{\substack{\mathfrak{LM}= \id \\ N(\mathfrak{L}) \le X}} \mu(\mathfrak{L}) N(\mathfrak{M}) \sum_{ \substack{(\fm, \mathfrak{M})= (1) \\ \fm= (\nu), \  \nu \gg 0 \\ N(\fm) \le Y}}
 \frac{1}{N(\fm)} \sum_{\epsilon \in U} \sum_{c \in \mathfrak{M}^{*}/U}
  \frac{S(\nu,  \epsilon^2; c)}{|N(c)|} \prod_{i=1}^n J_{2k-1} \left(  \frac{4\pi \sqrt{\nu^{(i)} } |\epsilon^{(i)}|}{|c^{(i)}|}\right).
\end{equation*}
Note that $$\sum_{\substack{\mathfrak{LM}= \id \\ N(\mathfrak{L}) \le X}} \mu(\mathfrak{L}) N(\mathfrak{M}) \ = \  N(\id) \prod_{\fp | \id} \left( 1- \frac{1}{N(\fp)}\right) + O \left( \frac{N(\id)^{1+\varepsilon}}{X}\right).$$
For all $x> 0$ we have $J_{2k-1}(x) \ll 1$. Moreover, from the integral representation (see \cite [8.411.10]{GR})
\begin{equation*}
J_{2k-1} (x)\ = \  \frac{1}{\Gamma(2k - \frac{1}{2}) \Gamma (\frac{1}{2})} \left( \frac{x}{2}\right)^{2k-1} \int_{-1}^{1} e^{ixt} (1-t^2)^{2k-3/2}dt
\end{equation*}
and Stirling's formula, we deduce
\begin{equation*}
J_{2k-1}(x)\ \ll\ \left( \frac{ex}{4k}\right)^{2k-1}.
\end{equation*}
Hence we have
\begin{equation} \label{JBesselEstimate}
J_{2k-1}(x)\ \ll\ \min \left\{ 1,  \left( \frac{ex}{4k}\right)^{2k-1} \right\}\  \ll\  \left( \frac{ex}{4k}\right)^{2k-1-\eta}  \quad \text{for} \ 0 \le \eta < 1.
\end{equation}
For $x =  \frac{4\pi \sqrt{\nu^{(i)}} |\epsilon^{(i)}|}{|c^{(i)}|}$, we choose $\eta =0$ if $|\epsilon^{(i)}| \le 1$ and $0 < \eta < 1$ if $|\epsilon^{(i)}| > 1$. Then we have
\begin{align} \label{ProdBessel}
\prod_{i=1}^{n} J_{2k-1} \left(  \frac{4\pi \sqrt{\nu^{(i)}} |\epsilon^{(i)}|}{|c^{(i)}|} \right)
&\ \ll\  \left( \frac{(e\pi)^n \sqrt{N(\nu)}}{k^n |N(c)|}\right)^{2k-1} |N(c)|^{\eta} \prod_{|\epsilon^{(i)}| > 1} |\epsilon^{(i)}|^{-\eta}  \nonumber \\
& \ \ll\   \frac{ \sqrt{N(\nu)}}{k^n |N(c)|} |N(c)|^{\eta} \prod_{|\epsilon^{(i)}| > 1} |\epsilon^{(i)}|^{-\eta}
\end{align}
provided $\displaystyle \frac{(e\pi)^n \sqrt{N(\nu)}}{k^n|N(c)|} \le 1$.
Note that $N(\nu) \le Y$ and $\displaystyle |N(c)| \geq N(\mathfrak{M}) \geq \frac{N(\id)}{X}$.
Hence $\displaystyle \frac{(e\pi)^n \sqrt{N(\nu)}}{k^n|N(c)|} \le \frac{(e\pi)^nXY^{1/2}}{k^nN(\id)}$.
By (\ref{ProdBessel}) and (\ref{WeilBound}), we have
\begin{align*}
E_{2k, \id} &\ \ll\ \sum_{\substack{\mathfrak{LM}= \id \\ N(\mathfrak{L}) \le X}}  N(\mathfrak{M}) \sum_{ \substack{(\fm, \mathfrak{M})= (1) \\ \fm= (\nu), \  \nu \gg 0 \\ N(\fm) \le Y}}
 \frac{1}{N(\fm)} \sum_{\epsilon \in U} \sum_{c \in \mathfrak{M}^{*}/U} \frac{|N(c)|^{\frac{1}{2}+ \varepsilon}}{|N(c)|} \frac{ \sqrt{N(\nu)}}{k^n |N(c)|} |N(c)|^{\eta} \prod_{|\epsilon^{(i)}| > 1} |\epsilon^{(i)}|^{-\eta} \\
 &\ \ll\ k^{-n} N(\id)^{-1/2 + \eta + \varepsilon} (XY)^{1/2}
\end{align*}
where for the last inequality we used (see \cite[p. 136]{Lu})
\begin{equation} \label{SumUnits}
\sum_{\epsilon \in U} \prod_{|\epsilon^{(i)}| > 1} |\epsilon^{(i)}|^{-\eta} \ <\ \infty.
\end{equation}
We take $X=Y^{1/2}= k^{\frac{2}{5}n}N(\id)^{\frac{3}{5}}$ and this complete the proof.
\end{proof}

\section{The explicit formula}
Let $R:= c_f=k^{2n}N(\id)$.
Following from \cite[section 4]{ILS}, we have the explicit formula
 \begin{multline*}
\sum_{\gamma_{f}} \phi \left(\frac{\gamma_{f}}{2\pi} \log R \right)
\ = \  \frac{\widehat{\phi}(0)}{\log R} \left( \log N(\id) + 2 \log D -2n \log 2\pi \right)  \\
+ \frac{2n}{\log R} \int_{-\infty}^{\infty} \frac{\Gamma'}{\Gamma} \left( k + \frac{2\pi i t}{\log R}\right) \phi (t) dt
- 2 \sum_{\fp} \sum_{\nu=1}^{\infty} \widehat{\phi} \left( \frac{\nu \log N(\fp)}{\log R}\right)
   \frac{a_{f}(\fp^{\nu}) \log N(\fp)}{N(\fp)^{\nu/2}\log R},
\end{multline*}
where $a_f(\fp^{\nu}) = \alpha (\fp)^{\nu} + \beta (\fp)^{\nu}$ for $\fp \nmid \id$ and $a_f(\fp^{\nu}) = \lambda_f^{\nu} (\fp)$ for $\fp | \id$.

\begin{prop} \label{densityformula}
Let $\phi$ be an even Schwartz function on $\mr$ whose Fourier transform $\widehat{\phi}$ has compact support. Then for $f \in H_{2k}^{\star} (\id)$ we have
\begin{equation}
D(f ; \phi ) \ = \  \widehat{\phi} (0) \frac{\log (k^{2n} N(\id))}{\log R} + \frac{1}{2} \phi (0) - P(f; \phi) + O \left(  \frac{\log \log (k^{2n} N(\id))}{\log R} \right),
\end{equation}
where
\begin{equation*}
P(f; \phi) \ := \  \sum_{\fp \nmid \id} \widehat{\phi} \left( \frac{\log N(\fp)}{\log R} \right) \frac{2 \lambda_f(\fp) \log N(\fp)}{N(\fp)^{1/2} \log R}.
\end{equation*}
\end{prop}

\begin{proof}
We have (see \cite[p.86]{ILS})
\begin{equation*}
\int_{-\infty}^{\infty} \frac{\Gamma'}{\Gamma} \left(  k + \frac{2\pi i t}{\log R} \right) \phi (t) dt \ = \   \widehat{\phi} (0) \log k + O(1).
\end{equation*}
Since $|\alpha (\fp)|, |\beta (\fp)| \le 1$, we have $|a_f(\fp^{\nu})| \le 2$ for all $\nu$. Hence for $\nu \geq 3$,
\begin{equation}
 \sum_{\fp} \sum_{\nu \geq 3} \widehat{\phi} \left( \frac{\nu \log N(\fp)}{\log R}\right) \frac{a_f(\fp^{\nu}) \log N(\fp)}{N(\fp)^{\nu/2} \log R} \ll \frac{1}{\log R}.
\end{equation}
Recall the relation $\alpha (\fp) + \beta (\fp)= \lambda_f(\fp), \ \alpha^2(\fp) + \beta^2(\fp)= \lambda_f(\fp^2) -1$ and (\ref{eigenvaluelevel}), we deduce
\begin{multline*}
D(f ; \phi ) \ = \  \widehat{\phi} (0) \frac{\log (k^{2n} N(\id))}{\log R}
- \sum_{\fp \nmid \id} \widehat{\phi} \left( \frac{\log N(\fp)}{\log R} \right) \frac{2 \lambda_f(\fp) \log N(\fp)}{N(\fp)^{1/2} \log R} \\
-\sum_{\fp \nmid \id} \widehat{\phi} \left( \frac{2\log N(\fp)}{\log R} \right) \frac{2 \lambda_f(\fp^2) \log N(\fp)}{N(\fp) \log R}
+ \sum_{\fp \nmid \id} \widehat{\phi} \left( \frac{2\log N(\fp)}{\log R} \right) \frac{2\log N(\fp)}{N(\fp) \log R} \\
+ O \left( \frac{1+ \log \log 3N(\id)}{\log R}\right).
\end{multline*}
By the Landau Prime Ideal Theorem (see \cite [Chapter XV \S 5]{La}) and partial summation,
\begin{equation*}
\sum_{\fp \nmid \id} \widehat{\phi} \left( \frac{2\log N(\fp)}{\log R} \right) \frac{2\log N(\fp)}{N(\fp) \log R}
\ = \  \frac{1}{2} \phi (0) + O \left(  \frac{1}{\log R}\right).
\end{equation*}
Assuming GRH for $L(s, \mathrm{sym}^2f)$, we have
\begin{equation*}
\sum_{\fp \nmid \id} \widehat{\phi} \left( \frac{2\log N(\fp)}{\log R} \right) \frac{2 \lambda_f(\fp^2) \log N(\fp)}{N(\fp) \log R}
\ \ll\ \frac{ \log \log (k^{2n}N(\id))}{\log R}.
\end{equation*}
The proposition follows from the above estimates.
\end{proof}

\section{Proof of the main theorem}
Let
\begin{equation*}
B_{2k}^{\star} (\phi) \ := \  \sum_{f \in H_{2k}^{\star} (\id)} D(f ; \phi).
\end{equation*}
By Proposition \ref{densityformula},
\begin{equation}
B_{2k}^{\star} (\phi)\ = \  |H_{2k}^{\star} (\id)| E(\phi) - P_{2k}^{\star} (\phi) + O \left(   |H_{2k}^{\star} (\id)|  \frac{\log \log (k^nN(\id))}{\log R}\right),
\end{equation}
where $E(\phi)= \widehat{\phi} (0) + \frac{1}{2} \phi (0)$ and
\begin{align*}
P_{2k}^{\star} (\phi) &\ = \  \sum_{\fp \nmid \id} \Delta_{2k, \id}^{\star} (\fp) \widehat{\phi} \left( \frac{\log N(\fp)}{\log R}\right) \frac{2 \log N(\fp)}{N(\fp)^{1/2} \log R} \\
  &\ = \  \sum_{\fp \nmid \id} (\Delta_{2k, \id}' (\fp) + \Delta_{2k, \id}^{\infty} (\fp) ) \widehat{\phi} \left( \frac{\log N(\fp)}{\log R}\right) \frac{2 \log N(\fp)}{N(\fp)^{1/2} \log R}.
\end{align*}
By Lemma \ref{Deltainfty} with $X=Y^{1/2} = (k^nN(\id))^{\delta}$ for some $\delta > 0$, we have
\begin{equation*}
\sum_{\fp \nmid \id} \Delta_{2k, \id}^{\infty} (\fp) \widehat{\phi} \left( \frac{\log N(\fp)}{\log R}\right) \frac{2 \log N(\fp)}{N(\fp)^{1/2} \log R} \ = \  o(k^nN(\id)).
\end{equation*}
By Proposition \ref{Petersson}, we have
\begin{align*}
M_{2k}^{\star} (\phi) \ := \  &\sum_{\fp \nmid \id} \Delta_{2k, \id}' (\fp) \widehat{\phi} \left( \frac{\log N(\fp)}{\log R}\right) \frac{2 \log N(\fp)}{N(\fp)^{1/2} \log R} \\
 \ = \ & \frac{2 \zeta_F(2) (-1)^{nk}}{(2\pi)^n} (2k-1)^n \sum_{\substack{\mathfrak{LM} = \id \\ N(\mathfrak{L}) \le X}} \mu (\mathfrak{L}) N(\mathfrak{M}) \sum_{ \substack{(\fm, \mathfrak{M})= (1) \\ \fm= (\nu), \  \nu \gg 0 \\ N(\fm) \le Y}} \frac{1}{N(\fm)} \sum_{\substack{\fp \nmid \id \\ \fp= (\mu), \ \mu \gg 0}} \sum_{\epsilon \in U} \sum_{c \in \mathfrak{M}^{*}/U} \\
& \frac{S(\nu^2, \mu \epsilon^2; c)}{|N(c)|} \prod_{i=1}^n J_{2k-1} \left(  \frac{4\pi \sqrt{(\nu^{(i)})^2 \mu^{(i)}} |\epsilon^{(i)}|}{|c^{(i)}|}\right) \widehat{\phi} \left( \frac{\log N(\fp)}{\log R}\right)  \frac{2 \log N(\fp)}{N(\fp)^{1/2} \log R}.
\end{align*}

For $\displaystyle x=  \frac{4\pi \sqrt{(\nu^{(i)})^2 \mu^{(i)}} |\epsilon^{(i)}|}{|c^{(i)}|}$ we choose $\eta =0$ if $|\epsilon^{(i)}| \le 1$ and $0 < \eta < 1$ if $|\epsilon^{(i)}| > 1$ in (\ref{JBesselEstimate}).
Thus
\begin{align} \label{Bessel}
\prod_{i=1}^{n} J_{2k-1} \left(  \frac{4\pi \sqrt{(\nu^{(i)})^2 \mu^{(i)}} |\epsilon^{(i)}|}{|c^{(i)}|} \right)
&\ \ll\  \left( \frac{(e\pi)^n \sqrt{N(\nu)^2 N(\mu)}}{k^n |N(c)|}\right)^{2k-1} |N(c)|^{\eta} \prod_{|\epsilon^{(i)}| > 1} |\epsilon^{(i)}|^{-\eta}  \nonumber \\
& \ \ll\   \frac{ \sqrt{N(\nu)^2 N(\mu)}}{k^n |N(c)|} |N(c)|^{\eta} \prod_{|\epsilon^{(i)}| > 1} |\epsilon^{(i)}|^{-\eta}
\end{align}
provided for $\displaystyle \frac{(e\pi)^n \sqrt{N(\nu)^2 N(\mu)}}{k^n |N(c)|} \le 1$.
Suppose $\widehat{\phi}$ has support in $(-u, u)$. So $N(\mu)= N(\fp) \le P=R^{u}$. Moreover $N(\nu)= N(\fm) \le Y$ and $\displaystyle |N(c)| \geq N(\mathfrak{M}) \geq \frac{N(\id)}{X}$, we have
$\displaystyle \frac{(e\pi)^n \sqrt{N(\nu)^2 N(\mu)}}{k^n |N(c)|} \le \frac{(e\pi)^n P^{1/2}XY}{k^nN(\id)}$. We choose $XY= (k^nN(\id))^{\delta}$ for some $\delta >0$. So the estimate is valid for
$\displaystyle u \le \frac{2(1-\delta) \log (k^nN(\id))}{\log (k^{2n}N(\id))}$.
By (\ref{WeilBound}) and (\ref{Bessel}), we have
\begin{align*}
M_{2k}^{\star} (\phi) \ \ll\ & (2k-1)^n \sum_{\substack{\mathfrak{LM}= \id \\ N(\mathfrak{L}) \le X}} N(\mathfrak{M}) \sum_{\substack{(\fm, \mathfrak{M}) =  1 \\ \fm= (\nu), \ \nu \gg 0 \\ N(\fm) \le Y}} \frac{1}{N(\fm)} \sum_{\substack{\fp \nmid \id \\ \fp= (\mu), \ \mu \gg 0}} \sum_{\epsilon \in U} \sum_{c \in \mathfrak{M}^{*}/U} \\
 & \frac{N((\fm^2, \fp, (c)))^{1/2} \tau((c)) (N(\fm)^2 N(\fp))^{1/2}}{k^n |N(c)|^{3/2-\eta}} \prod_{|\epsilon^{(i)}|>1} |\epsilon^{(i)}|^{-\eta} \frac{\log N(\fp)}{N(\fp)^{1/2} \log R} \\
 \ \ll\ & \frac{ N(\id)^{-1/2+ \eta + \varepsilon} XYP}{\log R},
\end{align*}
where for the last inequality we used (\ref{SumUnits}).

Hence $M_{2k}^{\star} (\phi)= o(k^n N(\id))$ for  $\displaystyle \frac{ N(\id)^{-1/2+ \eta + \varepsilon} XYP}{\log R} \le k^nN(\id)$.
By taking logarithms, we have
\begin{equation*}
u\ \le\ \left(\frac{3}{2} - (\delta + \eta + \varepsilon) \right) \frac{\log(k^nN(\id))}{\log (k^{2n}N(\id))} - \left( \frac{1}{2} -(\eta+ \varepsilon)\right) \frac{\log k^n}{\log (k^{2n}N(\id))}.
\end{equation*}



\end{document}